\documentclass{amsart}
\usepackage{amsmath,amsthm,amssymb,color}

\makeatletter
    
    \@addtoreset{equation}{section}
  \makeatother

\newtheorem{definition}{Definition}[section]
\newtheorem{proposition}[definition]{Proposition}
\newtheorem{theorem}[definition]{Theorem}
\newtheorem{lemma}[definition]{Lemma}
\newtheorem{corollary}[definition]{Corollary}
\newtheorem{remark}{Remark}[section]

\pagebreak

\title[Semilinear damped wave equation]{Global well-posedness for the
semilinear wave equation with time dependent damping in the overdamping case}

\author[M. Ikeda]{Masahiro Ikeda}
\address[M. Ikeda]{Department of Mathematics, Faculty of Science and Technology, Keio University, 3-14-1 Hiyoshi, Kohoku-ku, Yokohama, 223-8522, Japan/Center for Advanced Intelligence Project, RIKEN, Japan}
\email{masahiro.ikeda@keio.jp/masahiro.ikeda@riken.jp}

\author[Y. Wakasugi]{Yuta Wakasugi}
\address[Y. Wakasugi]{Department of Engineering for Production and Environment,
Graduate School of Science and Engineering,
Ehime University,
3 Bunkyo-cho, Matsuyama, Ehime, 790-8577,
Japan}%
\email{wakasugi.yuta.vi@ehime-u.ac.jp}
\begin{document}
\begin{abstract}
We study the global existence of solutions to the Cauchy problem for the wave equation with time-dependent damping and a power nonlinearity in the overdamping case:
\begin{align*}%
	\left\{\begin{array}{ll}
	u_{tt} - \Delta u + b(t) u_t = N(u),&t\in [0,T),\ x\in \mathbb{R}^d,\\
	u(0) = u_0,\ u_t(0) = u_1,&x\in \mathbb{R}^d.
	\end{array}\right.
\end{align*}%
Here,
$b(t)$ is a positive $C^1$-function on $[0,\infty)$ satisfying
\[	b(t)^{-1} \in L^1(0,\infty),
\]%
whose case is called {\em overdamping}.
$N(u)$ denotes the $p$-th order power nonlinearities.
It is well known that the problem is locally well-posed in the energy space
$H^1(\mathbb{R}^d)\times L^2(\mathbb{R}^d)$
in the energy-subcritical or energy-critical case $1\le p\le p_1$,
where $p_1:=1+\frac{4}{d-2}$ if $d\ge 3$ or $p_1=\infty$ if $d=1,2$.
It is known that when $N(u):=\pm|u|^p$, small data blow-up 
in $L^1$-framework occurs in the case $b(t)^{-1} \notin L^1(0,\infty)$ and $1<p<p_c(< p_1)$, 
where $p_c$ is a critical exponent, i.e. threshold exponent dividing
the small data global existence and the small data blow-up.
The main purpose in the present paper is to prove the global well-posedness
to the problem for small data $(u_0,u_1)\in H^1(\mathbb{R}^d)\times L^2(\mathbb{R}^d)$
in the whole energy-subcritical case, i.e. $1\le p<p_1$.
This result implies that the small data blow-up does not occur in the overdamping case,
different from the other case $b(t)^{-1}\notin L^1(0,\infty)$, i.e.
the effective or non-effective damping. 
\end{abstract}
\keywords{Global well-posedness, Energy-subcritical wave equations, Time-dependent damping, Overdamping}

\maketitle


\section{Introduction}
\footnote[0]{2010 Mathematics Subject Classification. 35L71; 35L15; 35A01}
\subsection{Background}
In the present paper, we study the global existence of solutions to
the Cauchy problem for the wave equation
with time-dependent damping and a power-type nonlinearity in the overdamping case:
\begin{align}%
\label{sdw}
	\left\{\begin{array}{ll}
	u_{tt} - \Delta u + b(t) u_t = N(u),&t\in [0,T), x\in \mathbb{R}^d,\\
	u(0) = u_0,\ u_t(0) = u_1,&x\in \mathbb{R}^d.
	\end{array}\right.
\end{align}%
Here, $d\in \mathbb{N}$, $T>0$, $u = u(t,x)$ is a real-valued unknown function of $(t,x)$,
$b=b(t)$ is a positive $C^1$-function of $t\in [0,\infty)$ satisfying
\begin{align}%
\label{b}
	b(t)^{-1} \in L^1(0,\infty)
	\quad\mbox{i.e.}\quad \int_0^{\infty}b(t)^{-1}dt<\infty,
\end{align}%
whose case is called the {\em overdamping}.
The nonlinear function $N:\mathbb{R}\mapsto \mathbb{R}$ satisfies the estimates
\begin{align}%
\label{n}
	N(0)=0 \quad \mbox{and}\quad
	|N(z)-N(w)| \le C_N (1+|z|+|w|)^{p-1}|z-w|
\end{align}%
for any $z,w\in \mathbb{R}$ with some constant $C_N\ge 0$ and $p\ge 1$. The exponent $p\ge 1$ belongs to $H^1$-subcritical or $H^1$-critical region, i.e.
\begin{align}%
\label{p}
	1\le p <p_1\ (d=1,2),\quad
	1\le p \le p_1 \ (d\ge 3),
	\ \mbox{where}\ 
	p_1:=
	\begin{cases}
	\infty,&\ \text{if}\ d=1,2,\\
	1+\frac{4}{d-2},&\ \text{if}\ d\ge 3,
	\end{cases}
\end{align}%
where $p_1$ is called the energy-critical or $H^1$-critical exponent.
Typical examples of $N(z)$ are the power-nonlinearities $\pm |z|^{p-1}z$ or $\pm |z|^p$ and the linear combination of $|z|^{q_1-1}z$ and $|z|^{q_2}$ with $1\le q_1,q_2\le p_1$.
The function $(u_0, u_1)$ is a prescribed function of $x\in \mathbb{R}^d$ belonging to the energy space, i.e.
\begin{align}%
\label{ini}
	(u_0,u_1) \in H^1(\mathbb{R}^d)\times L^2(\mathbb{R}^d),
\end{align}%
where $H^1(\mathbb{R}^d):=\{f\in L^2(\mathbb{R}^d);\|f\|_{H^1}^2:=\|f\|_{L^2}^2+\|\nabla f\|_{L^2}^2<\infty\}$.

Wave equations with a dissipative term are physical models describing the voltage and the current on an electrical transmission line with a resistance. The term $b(t)u_t$ is called the damping term, which prevents the motion of the wave and reduces its energy, and the coefficient $b(t)$ represents the strength of the damping. 

There are also large amount of literature of mathematical results about global existence of solutions, asymptotic behaviors of solutions, and blow-up phenomena to (\ref{sdw}) in the opposite case 
\[
 b(t)^{-1}\notin L^1(0,\infty)
\] 
to (\ref{b}), whose case is called effective damping or non-effective damping
(see \cite{LiZh95, TY1, Zh01, Ni03MathZ, N2, LNZ12, DaLu13, DaLuRe13, W14, Da15, DaLuRe15,
IkeOg16, IkInpre2, Wak16, W17, IkeSopre, LTW17, IkInMoYwpre} and the references therein).
However, there has been no result about the global existence of solutions to (\ref{sdw})
in the overdamping case (\ref{b}), especially there has not been known
whether local energy solutions to (\ref{sdw}) can be extended globally or not in the overdamping case (\ref{b}). 

Our aim in the present paper is to prove the global well-posedness to (\ref{sdw}) for small data in the energy space in the whole energy-subcritical case $1\le p<p_1$ and the overdamping case (\ref{b}), which implies that small data blow-up does not occur in this case different from the other case $b(t)^{-1}\notin L^1(0,\infty)$ (see Remark \ref{kre} for known results in the effective damping or non-effective damping case). The key idea of the proof of the global well-posedness is to derive an uniform bound of $L^2$-norm of local solutions to (\ref{sdw}) by using the overdamping condition (\ref{b}) (see (\ref{L^2}) more precisely). 
We also prove that besides the assumptions, under the defocusing condition, i.e.
\begin{equation}
\label{de1}
         \int_0^z N(s)ds\le 0 \quad \mbox{for any}\  z\in \mathbb{R},
\end{equation}
the smallness condition on the data can be removed.
A typical example satisfying \eqref{de1} is
$N(z) = - |z|^{p-1}z$.
We also prove a large data blow-up result to (\ref{sdw})
for the focusing nonlinearity
$N(z)=\pm |z|^{p}$ and $p\in (1,p_1]$ for suitable data $(u_0,u_1)\in H^1(\mathbb{R}^d)\times L^2(\mathbb{R}^d)$ which have a singularity at the origin $x=0$.
This implies that the smallness condition on the data is needed to prove the global existence of solutions to (\ref{sdw}) with $N(z)=\pm |z|^p$ in the case $p\in (1,p_1]$.
Moreover, we prove the non-existence of local weak solutions to (\ref{sdw}) with $N(z)=\pm |z|^p$ and $p>p_1$ for suitable data $(u_0,u_1)\in H^1(\mathbb{R}^d)\times L^2(\mathbb{R}^d)$ which have a singularity at the origin $x=0$. The definition of the weak solution is given in Definition \ref{def2-1}.



Now we state local well-posedness (LWP) result to (\ref{sdw}) in the energy space in the energy-subcritical case:
\begin{proposition}[LWP in the energy-subcritical case]
\label{prop_le}
Let $d\in \mathbb{N}$, $p\in [1,p_1)$, $b=b(t)$ be a positive $C^1$-function on $[0,\infty)$ and the nonlinear function $N:\mathbb{R}\mapsto \mathbb{R}$ satisfy the estimate \eqref{n},
and let $(u_0, u_1)\in H^1(\mathbb{R}^d)\times L^2(\mathbb{R}^d)$. Then the following statements hold:
\begin{itemize}
\item (Existence) If $\|(u_0,u_1)\|_{H^1\times L^2}\le M$ with some $M>0$, then there exists a positive constant $T=T(M)>0$ such that there exists a unique solution 
\[
  u\in C([0,T); H^1(\mathbb{R}^d))\cap C^1([0,T); L^2(\mathbb{R}^d))
\] to \eqref{sdw} on $[0,T)$.
\item (Uniqueness) Let $u$ be the solution to (\ref{sdw}) obtained in the Existence part.
Let $T_1\in (0,T(M)]$ and
$v\in C([0,T_1); H^1(\mathbb{R}^d)) \cap C^1([0,T_1); L^2(\mathbb{R}^d))$
be another solution to (\ref{sdw}) on $[0,T_1)$. If $(v(0),\partial_tv(0))=(u_0,u_1)$, then $u|_{I_{T_1}}=v$ on $[0,T_1)$.
\item (Continuity of the flow map) Let $M>0$ and $T=T(M)$
be the same as in the Existence part. Then the flow map 
\begin{align*}
   &\Xi:\{(u_0,u_1)\in H^1(\mathbb{R}^d)\times L^2(\mathbb{R}^d);\|(u_0,u_1)\|_{H^1\times L^2}\le M\}\\
   &\to C([0,T); H^1(\mathbb{R}^d))\cap C^1([0,T); L^2(\mathbb{R}^d)),\ \ \Xi[(u_0,u_1)](t)=u(t)
\end{align*}
is Lipschitz continuous.
\end{itemize}
By the existence result and the uniqueness result, the maximal existence time $T_{+}$ of the solution, which is called lifespan, is well defined as
\begin{align*}
      T_+:=\sup\{T\in (0,\infty]; \text{there exists a unique solution}\ u\ \text{to (\ref{sdw}) on }\ [0,T)\}.
\end{align*}
Moreover, the energy identity and the blow-up criterion hold:
\begin{itemize}
\item (Energy identity) The identity
\begin{align}
\label{energy}
     \frac{d}{dt}\left\{\frac{1}{2}\|\partial_t u(t)\|_{L^2}^2+\frac{1}{2}\|\nabla u(t)\|_{L^2}^2-\int_{\mathbb{R}^d}\tilde{N}(u(t,x))dx\right\}=-b(t)\|\partial_tu(t)\|_{L^2}^2
\end{align}
holds for any $t\in [0,T_+)$, where the function $\tilde{N}:\mathbb{R}\mapsto \mathbb{R}$ is defined by
\begin{equation}
\label{1-7-4}
        \tilde{N}(z):=\int_0^zN(s)ds.
\end{equation}
\item (Blow-up criterion) If $T_{+}<\infty$, then 
\[
   \liminf_{t\to T_{+}-0}\|(u (t),\partial_tu(t))\|_{H^1\times L^2}=\infty.
\]
\end{itemize}
\end{proposition}
In order to prove this proposition, by using an appropriate changing variable, we convert the equation  into another nonlinear wave equation, whose nonlinearity does not include any derivatives of the unknown function and satisfies the estimate (\ref{n}) (see Appendix \ref{a-1}). As the result, we can apply the local well-posedness result for nonlinear wave equations in the energy space in the energy-subcritical case. More precisely, in the case, $1\le p <\infty\ (d=1,2)$ or $1\le p \le \frac{d}{d-2}\ (d\ge 3)$, by the Sobolev embedding $H^1(\mathbb{R}^d)\subset L^{2p}(\mathbb{R}^d)$, it is easy to see that
there exist $T=T(\|(u_0,u_1)\|_{H^1\times L^2})>0$ and
a unique solution $u$ of the Cauchy problem
\eqref{sdw} on $[0,T)$. In the whole $H^1$-subcritical case, i.e. $\frac{1}{d-2}=1+\frac{2}{d-2}<p<1+\frac{4}{d-2}\ (d\ge 3)$, the local well-posedness to a nonlinear wave equation in $H^1(\mathbb{R}^d)\times L^2(\mathbb{R}^d)$ is proved in \cite{GiVe89, Ka92, KeTa98, MNO02, MN02, KiStVi14},
via the Strichartz estimates for the free Klein-Gordon evolution operator $\left\{e^{it\sqrt{1-\Delta}}\right\}_{t\in \mathbb{R}}$ and the real and complex interpolation theory.

\begin{remark}[LWP in the energy-critical case]
In the energy-critical case $p=p_1$, local well-posedness in $H^1(\mathbb{R}^d)\times L^2(\mathbb{R}^d)$ to (\ref{sdw}) also holds for arbitrary initial data in the energy space (see Theorem 2.7 in \cite{KM08} for example). The main difference from the subcritical case is that the existence time may depend not only on $\|(u_0,u_1)\|_{H^1\times L^2}$ but also on the profile of it, and the blow-up criterion is
\begin{equation}
\label{bcri}
       T_+<\infty\Longrightarrow \|u\|_{L_{t,x}^{\frac{2(d+1)}{d-2}}([0,T_+)\times \mathbb{R}^d)}=\infty.
\end{equation}
\end{remark}


\subsection{Main results}
We state our main result in the present paper, which gives global well-posedness (GWP) to (\ref{sdw}) in the energy space in the energy subcritical and the overdamping case:
\begin{theorem}[GWP in the energy space in the energy subcritical and the overdamping case]
\label{thm_ge}
Besides the assumptions of Proposition \ref{prop_le}, we assume that $b=b(t)$ satisfies the overdamping condition (\ref{b}). Then, there exists $\varepsilon_0 > 0$ depending only on $d,p,\|b^{-1}\|_{L_t^1(0,\infty)}, C_N$ such that for any $\varepsilon\in [0,\varepsilon_0]$, if the initial data $(u_0,u_1)\in H^1(\mathbb{R}^d)\times L^2(\mathbb{R}^d)$ satisfies
$\| (u_0, u_1) \|_{H^1 \times L^2} \le \varepsilon$, then the maximal existence time $T_+$, which is defined in Proposition \ref{prop_le}, is infinity, i.e. $T_+=\infty$, and the solution $u$ satisfies the estimate
\begin{equation}
\label{1-7-1}
           \sup_{t\in [0,\infty)}\|(u(t),\partial_tu(t))\|_{H^1\times L^2}\le C_0^*\varepsilon,
\end{equation}
where $C_0^*$ is a constant depending only on $d,p,\|b^{-1}\|_{L^1(0,\infty)}, C_N$.
\end{theorem}

\begin{remark}[Difference of our theorem with effective or non-effective damping case]
\label{kre}
When the nonlinearity is focusing, i.e.
$N(z)=\pm|z|^p$ and $b(t)^{-1}\notin L^1(0,\infty)$,
small data blow-up for suitable data $(u_0,u_1)\in (L^1(\mathbb{R}^d))^2$ occurs for any $p\in (1,p_c]$, where $p_c$ is a critical exponent, whose definition is the threshold exponent dividing small data global existence and small data blow-up. It is interesting that $p_c=p_c(b)$ changes with regard to the decay rate of the function $b(t)$ as $t\rightarrow\infty$ 
but is less than the $H^1$-critical exponent in any case. More precisely, in the case of $b(t)\equiv 1$, i.e. the classical damping case, it is known that $p_c$ is given by the Fujita exponent $p_F$, 
where $p_F=p_F(d):=1+\frac{2}{d}<1+\frac{4}{d-2}=:p_1$ (see \cite{LiZh95, TY1, N2, Zh01, IT, HKI} and the references therein). In the case of $b(t)\equiv 0$, i.e. 
when there is no dissipative term in (\ref{sdw}), it is also well known that $p_c$ is given by the Strauss exponent $p_S$ (see \cite{YZ06, Zhou07} and the references therein), where $p_S=p_S(d)$ with $d\ge 2$ is defined as the positive root of the quadratic equation
\[
     (d-1)p^2-(d+1)p-2=0,
\]
that is
\[
     p_S(d):=\frac{d+1+\sqrt{d^2+10d-7}}{2(d-1)}<p_1.
\]
In the case of $b(t):=\mu (1+t)^{-\beta}$ with $\mu>0$ and $\beta\in [-1,1)$, whose case is called effective damping, it is proved in \cite{LNZ12, W17, IkInpre2} that $p_c$ is given by the Fujita exponent $p_F$. In the scaling invariant damping case, i.e. $b(t):=\frac{\mu}{1+t}$ with $\mu>0$, the situation is very delicate and critical exponent changes with regard to the size of the coefficient $\mu$
(see \cite{W14, Da15, DaLuRe15, Wak16, LTW17, IkeSopre, TuLi} and the references therein), but the known critical exponents are less than the $H^1$-critical exponent.
In the case of $b(t):=\mu(1+t)^{-\beta}$ with $\beta>1$ and $\mu>0$, it is expected that the critical exponent $p_c$ is given by the Strauss exponent $p_S$ (see \cite{LTpre}). We note that $p_S$ is less than the $H^1$-critical exponent as stated above.
\end{remark}


Under the defocusing condition (\ref{de}), the smallness condition on the data in Theorem \ref{thm_ge} can be removed:
\begin{corollary}[Large data global well-posedness in the defocusing case]
\label{cor1}
Besides the same assumptions in Theorem \ref{thm_ge}, we assume that the nonlinear function $N$ satisfies the defocusing condition
\begin{equation}
\label{de}
                 \tilde{N}(z) \le 0 \quad \mbox{for any}\ z\in \mathbb{R},
\end{equation}
where $\tilde{N}$ is defined by (\ref{1-7-4}). Then the maximal existence time $T_+$ is infinity and the estimate
\begin{equation}
\label{1-9-3}
          \sup_{t\in [0,\infty)}\|(u(t),\partial_tu(t))\|_{H^1\times L^2}\le C_1^*(\|(u_0,u_1)\|_{H^1\times L^2}+\|u_0\|_{L^{p+1}}^{\frac{p+1}{2}}),
\end{equation}
holds, where $C_1^*$ is a constant depending only on $p,d,\|b^{-1}\|_{L^1(0,\infty)}$ and $C_N$.
\end{corollary}
A typical example satisfying the defocusing condition \eqref{de}
is $N(z) = - |z|^{p-1}z$. Indeed, for this $N(z)$ we have
$\tilde{N}(z) = -(p+1)^{-1}|z|^{p+1}\le 0$.


On the other hand, we next  show that
in general
the smallness condition on the data in Theorem \ref{thm_ge} cannot be removed. Indeed, large data blow-up for suitable $(u_0,u_1)\in H^1(\mathbb{R}^d)\times L^2(\mathbb{R}^d)$ which has a singularity at the origin is proved
for the focusing nonlinearity $N(z):=\pm |z|^p$:
\begin{theorem}[Large data blow-up in the energy-subcritical or critical
case for the focusing nonlinearity $N(z)=\pm|z|^p$]
\label{thm_lb}
Let $d\in \mathbb{N}$, $b=b(t)$ be a positive $C^1$-function on $[0,\infty)$, $p>1$ (if $d=1,2$), $p\in (1,p_1]$ (if $d\ge 3$), $N(z)=\pm |z|^p$ and $(a_0,a_1)\in H^1(\mathbb{R}^d)\times L^2(\mathbb{R}^d)$. We assume that the function $(a_0,a_1)$ satisfies
\begin{align}
\label{bc}
        \pm(b(0)a_0+a_1)(x)\ge
        \left\{\begin{array}{ll}
        |x|^{-k},&\text{if}\ |x|\le 1,\\
        0,&\text{if}\  |x|>1,
        \end{array}\right.
\end{align}
with $k<\frac{d}{2}$, where double-sign corresponds. Then there exist constants $\lambda_0$ and $C_2^*>0$ depending only on $d,p,k,\|b\|_{W^{1,\infty}(0,2)}$ such that for any $\lambda\in (\lambda_0,\infty)$, the lifespan $T_+$ defined in Proposition \ref{prop_le} with $(u_0,u_1):=\lambda(a_0,a_1)$ is estimated by
\begin{equation}
\label{Btime}
          T_+\le C_2^*\lambda^{-\frac{1}{\frac{p+1}{p-1}-k}},
\end{equation}
which with the blow-up criterion in Proposition \ref{prop_le} and (\ref{bcri}) implies
\begin{align*}
        \liminf_{t\rightarrow T_+-0}\|(u(t),\partial_t u(t))\|_{H^1\times L^2}&=\infty,\ \ \text{if}\ 1<p<p_1, \\
        \|u\|_{L_{t,x}^{\frac{2(d+1)}{d-2}}([0,T_+)\times\mathbb{R}^d)}&=\infty,\ \ \text{if}\ \  p=p_1.
\end{align*}
\end{theorem}

The proof of this proposition is based on the
so called test-function method,
which can be found in Theorem 2.3 in \cite{IkIn15} or Theorem 2.4 in \cite{IkInpre}
(see also \cite{FO16, FIW}).

Next we state non-existence of local solutions
for the focusing nonlinearity $N(z)=\pm |z|^p$
with the $H^1$-supercritical exponent $p>p_1$,
for suitable $(u_0,u_1)\in H^1(\mathbb{R}^d)\times L^2(\mathbb{R}^d)$ which have a singularity at the origin $x=0$.
\begin{theorem}[Non-existence of local solutions in the energy-supercritical case
for the focusing nonlinearity $N(z)= \pm |z|^p$]
\label{thm_non}
Let $d\in \mathbb{N}$ with $d\ge 3$, $b=b(t)$ be a positive $C^1$-function on $[0,\infty)$, $p>p_1$, $N(z)=\pm |z|^p$ and $(u_0,u_1)\in H^1(\mathbb{R}^d)\times L^2(\mathbb{R}^d)$. We assume that the function $(u_0,u_1)$ satisfies the estimate (\ref{bc}) with $k\in (\frac{p+1}{p-1},\frac{d}{2})$, where double-sign corresponds. Then for any $T>0$, there does not exist weak solution to (\ref{sdw}) on $[0,T)$, where the weak solution is defined by Definition \ref{def2-1}.
\end{theorem}
The proof of this theorem is based on the argument of the proof of Theorem 2.5 in \cite{IkIn15} or Theorem 2.7 in \cite{IkInpre}, though different from the proof of Theorem 2.5 in \cite{IkIn15}, we have to deal with the time-dependent coefficient $b(t)$.


\section{Proof of global well-posedness in the energy subcritical case}
In this section, we give a proof of Theorem \ref{thm_ge}. The key trick in the proof is deriving the uniform upper estimate of $L^2$-norm of local solutions $u$ by using the assumption that $b$ satisfies the overdamping condition (\ref{b}) (see (\ref{L^2}) more precisely).
\begin{proof}[Proof of Theorem \ref{thm_ge}]
By the local well-posedness result (Proposition \ref{prop_le}), we can find a unique solution $u\in C([0,T_+);H^1(\mathbb{R}^d))\cap C^1([0,T_+);L^2(\mathbb{R}^d))$ to (\ref{sdw}) on $[0,T_+)$.
Thus we can define a continuous function $Q:[0,T_+)\mapsto [0,\infty)$ as
\begin{align}%
\label{m}
	Q(t) = \sup_{0\le s \le t} \| (u(s),u_t(s)) \|_{H^1\times L^2}.
\end{align}%
We will prove that there exist constants
$\varepsilon_0 = \varepsilon_0(d, p, \|b^{-1}\|_{L^1(0,\infty)},C_N) > 0$
and
$C_{0}^* = C_{0}^*(d, p, \|b^{-1}\|_{L^1(0,\infty)}, C_N) > 0$ such that for any
$\varepsilon\in [0,\varepsilon_0)$,
if $\| (u_0, u_1) \|_{H^1\times L^2} \le \varepsilon$, then
\begin{align}%
\label{apr}
	Q(t) \le C_{0}^* \| (u_0, u_1) \|_{H^1\times L^2}\le C_0^*\varepsilon<\infty.
\end{align}%
for any $t\in [0,T_+)$. Then by the blow-up criterion in Proposition \ref{prop_le} with the estimate (\ref{apr}), we see that $T_+=\infty$ and the estimate (\ref{1-7-1}) holds, which completes the proof of the theorem.

In order to prove the estimate (\ref{apr}), integrating the energy identity (\ref{energy}),
we have
\begin{align}%
\label{ee}
	&\frac{1}{2}\|\partial_tu(t)\|_{L^2}^2+\frac{1}{2}\|\nabla u(t)\|_{L^2}^2
	+ \int_0^t b(s) \| u_t (s) \|_{L^2}^2 ds \\
\nonumber
&\quad
	=\frac{1}{2}\|\nabla u_0\|_{L^2}^2+\frac{1}{2}\|u_1\|_{L^2}^2
	-\int_{\mathbb{R}^d}\tilde{N}(u_0(x))dx+\int_{\mathbb{R}^d}\tilde{N}(u(t,x))dx\\
	&\quad\le\frac{1}{2}\|\nabla u_0\|_{L^2}^2+\frac{1}{2}\|u_1\|_{L^2}^2
	+C_N\|u_0\|_{L^{p+1}}^{p+1}+C_N\|u(t)\|_{L^{p+1}}^{p+1} \notag
\end{align}%
for any $t\in [0,T_+)$, where we have used the estimate (\ref{n}). This implies that the estimate
\begin{equation}
\label{2-4-2}
    \int_0^t b(s) \| u_t (s) \|_{L^2}^2 ds
    \le C_1\left\{ \| (u_0, u_1) \|_{\dot{H}^1\times L^2}^2 + \| u_0 \|_{L^{p+1}}^{p+1}
	+ \| u(t) \|_{L^{p+1}}^{p+1} \right\}
\end{equation}
holds for any $t\in [0,T_+)$, where $C_1$ is defined by $C_1:=\max\left(\frac{1}{2},C_N\right)$.
By the fundamental theorem of calculus, Schwarz's inequality and by the overdamping condition \eqref{b}, the Fubini-Tonelli theorem and the estimate (\ref{2-4-2}), the estimates
\begin{align}%
\label{L^2}
	\| u(t) \|_{L^2}^2
	&= \left\| u_0 + \int_0^t u_t(s) ds \right\|_{L^2}^2\\
	&\le 2 \| u_0 \|_{L^2}^2
		+ 2 \left\| \left( \int_0^t b(s)^{-1}ds \right)^{1/2}
				\left( \int_0^t b(s) u_t(s)^2 ds \right)^{1/2} \right\|_{L^2}^2 \notag\\
	&\le 2 \| u_0 \|_{L^2}^2
		+ 2 \| b^{-1} \|_{L^1(0,\infty)} \int_0^t b(s) \| u_t (s) \|_{L^2}^2 ds \notag\\
	&\le	 C_2
		\left\{ \| (u_0, u_1) \|_{H^1\times L^2}^2 + \| u_0 \|_{L^{p+1}}^{p+1}
	+ \| u(t) \|_{L^{p+1}}^{p+1} \right\}
\notag
\end{align}%
hold for any $t\in [0,T_+)$, where $C_2$ is defined by $C_2:=\max\left(2,2\|b^{-1}\|_{L^1(0,\infty)}C_1\right)$.
Since $p$ belongs to the $H^1$-subcritical region \eqref{p},
the Sobolev embedding $H^1(\mathbb{R}^d)\subset L^{p+1}(\mathbb{R}^d)$ gives
\begin{equation}%
\label{sob}
	\| u(t) \|_{L^{p+1}} \le C_3 \| u(t) \|_{H^1} \le C_3 Q(t),
\end{equation}%
for any $t\in [0,T_+)$, where $C_3$ depends only on $p$ and $d$. Therefore, by combining the estimates (\ref{ee}), (\ref{L^2}) and (\ref{sob}), the inequality
\begin{align}%
\label{ap}
	Q(t) &\le (1+C_2)^{\frac{1}{2}}Q(0)+
	\{C_3^{p+1}(C_2+2C_N)\}^{\frac{1}{2}}Q(0)^{\frac{p+1}{2}}\\
	 &+ \{C_3^{p+1}(C_2+2C_N)\}^{\frac{1}{2}}Q(t)^{\frac{p+1}{2}}\notag\\
	 &\le C_4\left\{Q(0)+Q(0)^{\frac{p+1}{2}}+Q(t)^{\frac{p+1}{2}}\right\}
\notag
\end{align}%
holds for any $t\in [0,T_+)$, where $C_4:=\max\left((1+C_2)^{\frac{1}{2}}, \{C_3^{p+1}(C_2+2C_N)\}^{\frac{1}{2}})\}\right)$. Here we take $\varepsilon_0=\varepsilon_0(C_4,p)>0$ sufficiently small. Then by the standard continuity argument, for $\varepsilon\in[0,\varepsilon_0]$, we can derive the desired estimate \eqref{apr} if $\| (u_0, u_1) \|_{H^1\times L^2} \le \varepsilon$.
\end{proof}

\begin{proof}[Proof of Corollary \ref{cor1}]
In the same manner as the proof of the estimate (\ref{ee}), by the defocusing condition (\ref{de}), the estimate
\begin{align*}
	&\frac{1}{2}\|\partial_tu(t)\|_{L^2}^2+\frac{1}{2}\|\nabla u(t)\|_{L^2}^2
	+ \int_0^t b(s) \| u_t (s) \|_{L^2}^2 ds \\
&\quad
	=\frac{1}{2}\|\nabla u_0\|_{L^2}^2+\frac{1}{2}\|u_1\|_{L^2}^2
	-\int_{\mathbb{R}^d}\tilde{N}(u_0(x))dx+\int_{\mathbb{R}^d}\tilde{N}(u(t,x))dx\\
	&\quad\le\frac{1}{2}\|\nabla u_0\|_{L^2}^2+\frac{1}{2}\|u_1\|_{L^2}^2
	+C_N\|u_0\|_{L^{p+1}}^{p+1} \notag
\end{align*}
holds for any $t\in [0,T_+)$. In the same manner as the proof of the estimate (\ref{ap}), the inequality
\[
    Q(t) \le C_4\left\{\|(u_0,u_1)\|_{H^1\times L^2}+\|u_0\|_{L^{p+1}}^{p+1}\right\}
\]
is true for $t\in [0,T_+)$. By this estimate and the blow-up criterion in Proposition \ref{prop_le}, we find that 
$T_+=\infty$, and the estimate (\ref{1-9-3}) holds.
 which completes the proof of the corollary.
\end{proof}


\section{Proof of large data blow-up and non-existence of local solution}
In this section, we give a proof of the large data blow-up
of the solution
to (\ref{sdw}) in the energy-subcritical or critical case (Theorem \ref{thm_lb}),
and the non-existence of the local weak solution in the energy-supercritical case (Theorem \ref{thm_non}). 
Our proof for Theorem \ref{thm_lb} is based on that of Theorem 2.3 in \cite{IkIn15} or Theorem 2.4 in \cite{IkInpre}.
The proof of Theorem \ref{thm_non} is based on that of Theorem 2.5 in \cite{IkIn15} or Theorem 2.7 in \cite{IkInpre}.

We reduce the problems into whether there exists a weak solution to (\ref{sdw}) or not. The weak solution to (\ref{sdw}) is defined as follows.
\begin{definition}
\label{def2-1}
We say that $u$ is a weak solution to (\ref{sdw}) on $[0,T)$, 
if $u$ belongs to $L^p_{loc}([0,T)\times\mathbb{R}^d)$ and the identity
\begin{align}
\label{2-1-1}
	&\int_{[0,T)\times \mathbb{R}^d}
		u(t,x)\{(\partial^2_t\psi)(t,x)-(\Delta\psi)(t,x)-b'(t)\psi(t,x)-b(t)(\partial_t\psi)(t,x)\}dxdt\\
	&=\int_{\mathbb{R}^d}u_0(x)(\partial_t\psi)(0,x)dx+\int_{\mathbb{R}^d}
	\{b(0)u_0(x)+u_1(x)\}\psi(0,x)dx\notag\\
	&\ \ \ +\int_{[0,T)\times\mathbb{R}^d}N(u(t,x))\psi(t,x)dxdt\notag
\end{align}
holds for any $\psi\in C_0^{\infty}([0,T)\times \mathbb{R}^d)$.
We also define lifespan of the weak solution as
\[
            T_w:=\sup\{T\in (0,\infty]; \text{there exists a unique weak solution}\  u\  \text{to}\  \text{(\ref{sdw})}\ \text{on}\ [0,T)\}.
\]
\end{definition}

In order to prove Theorem \ref{thm_lb}, we first show that energy solution to (\ref{sdw}) on $[0,T)$ becomes weak solution to (\ref{sdw}) on $[0,T)$:
\begin{lemma}
\label{lemso}
Under the same assumptions as in Proposition \ref{prop_le}, let $u$ be a solution to (\ref{sdw}) on $[0,T)$. Then $u$ becomes
a weak solution to (\ref{sdw}) on $[0,T)$, which implies the estimate $T_+\le T_{w}$ holds.
\end{lemma}
The proof of this lemma is due to a standard density argument. We omit the detail
(see Proposition 4.2 in \cite{IkInYw17} or
Proposition 9.6 in \cite{W14-2} more precisely).

\begin{proposition}[Non-existence of global weak solution for large data
for the focusing nonlinearity $N(z)=\pm |z|^p$ with $p>1$]
\label{pro_ngw}
Let $d\in \mathbb{N}$, $b=b(t)$ be a positive $C^1$-function on $[0,\infty)$, $p>1$, 
$N(z)=\pm |z|^p$ and $(a_0,a_1)\in (L^1_{loc}(\mathbb{R}^d))^2$. We assume that the function 
$(a_0,a_1)$ satisfies (\ref{bc}) with $k<\min\left(d,\frac{p+1}{p-1}\right)$, where double-sign corresponds. 
Then there exists $\lambda_0>0$ and $C_2^*>0$ depending only on $d,p,k,\|b\|_{W^{1,\infty}(0,2)}$ such that for any $\lambda\in (\lambda_0,\infty)$, the maximal existence time $T_{w}$ with $(u_0,u_1)=\lambda(a_0,a_1)$ is estimated as
\begin{equation}
\label{Tes}
           T_{w}\le C_2^*\lambda^{-\frac{1}{\frac{p+1}{p-1}-k}}.
\end{equation}
\end{proposition}

In order to prove Theorem \ref{thm_non} (non-existence of local solutions in the energy space in the $H^1$-supercritical case), we prove the following non-existence result for local weak solutions with $p>1+\frac{2}{d-1}$ for suitable data $(u_0,u_1)\in (L^1_{loc}(\mathbb{R}^d))^2$. 

\begin{proposition}[Non-existence of local weak solution in $(L^1_{loc}(\mathbb{R}^d))^2$ data setting]
\label{pro_non}
Let $d\in \mathbb{N}$ with $d\ge 3$, $b=b(t)$ be a positive $C^1$-function on $[0,\infty)$, $p>1+\frac{2}{d-1}$, $N(z)=\pm |z|^p$ and $(u_0,u_1)\in (L^1_{loc}(\mathbb{R}^d))^2$. We assume that the function $(u_0,u_1)$ satisfies the estimate (\ref{bc}) with $k\in (\frac{p+1}{p-1},d)$, where double-sign corresponds. Then for any $T>0$, there does not exist weak solution to (\ref{sdw}) on $[0,T)$.
\end{proposition}

\subsection{Integral inequalities via a test-function method}
In this subsection, we derive two useful inequalities 
(Lemmas \ref{Lem 3-1}, \ref{Lem 3-2} below) by using suitable test-functions. We define the two functions $\eta=\eta(t)\in C_0^{\infty}([0,\infty))$, $\phi=\phi(x)\in C_0^{\infty}(\mathbb{R}^d)$
such as $0\le\eta,\phi\le 1$ and
\begin{equation}
\label{test}
        \eta (t):=\begin{cases}
                      1        & (0 \le t<1/2), \\ 
          \text{smooth} & (1/2\le t\le 1), \\ 
                      0        & (t> 1),
                     \end{cases}\\\
       \ \ \ \phi (x):=\begin{cases}
                     1        & (0 \le |x| <1/2), \\ 
         \text{smooth} & (1/2\le |x| \le 1), \\ 
                     0        & (|x|> 1).
                      \end{cases}
\end{equation}
For $\tau>0$, which will be chosen appropriately later, we also define the function $\psi_{\tau}$ of 
$(t,x)\in [0,\infty)\times \mathbb{R}^d$ such as
$$
	\psi_{\tau}=\psi_{\tau}(t,x):=\eta_{\tau}(t)\phi_{\tau}(x):=\eta(t/\tau)\phi(x/\tau).
$$
We define the open ball centered at the origin with radius $r>0$ in $\mathbb{R}^d$ 
by $B(r)$, i.e.
\[
B(r):=\{x\in\mathbb{R}^d; |x|<r \}.
\]
An upper estimate of the data can be derived as follows.
\begin{lemma}
\label{Lem 3-1}
Let $d\in \mathbb{N}$, $b=b(t)$ be a positive $C^1$-function on $[0,\infty)$, $p>1$, $N(z)=\pm|z|^p$, $q:=p/(p-1)$, $l\in \mathbb{N}$ with $l\ge 2q+1$, $(u_0,u_1)\in (L^1_{loc}(\mathbb{R}^d))^2$, $T>0$ and $u$ be a weak solution to (\ref{sdw}) on $[0,T)$. 
Then there exists a constant $C_3^*>0$ depending only on $d,p,l$, such that the estimate
\begin{equation}
\label{3-0-1}
\pm \int_{B(\tau)}(b(0)u_0+u_1)(x)\phi^l_{\tau}(x)dx\le C_3^*\tau^{d+1}\left\{\tau^{-2q}+\|b'\|_{L^{\infty}(0,T)}^q+\|b\|_{L^{\infty}(0,T)}^q\tau^{-q}\right\}, 
\end{equation}
holds for any $\tau\in (0,T)$, where the double-sign corresponds.
\end{lemma}

\begin{proof}[Proof of Lemma \ref{Lem 3-1}]
We only consider the case of $N(z)=|z|^p$, since we can treat the case of $N(z)=-|z|^p$ in a similar manner. Since $u\in L^p_{loc}([0,T)\times\mathbb{R}^d)$ and
$(u_0,u_1)\in (L^1_{loc}(\mathbb{R}^d))^2$, we can define the following two functions of $\tau\in (0,T)$
\begin{align*}
	I(\tau)&:=\int_{[0,\tau)\times B(\tau)}|u(t,x)|^p\psi^l_{\tau}(t,x)dxdt,\\
	J(\tau)&:= \int_{B(\tau)}(b(0)u_0(x)+u_1(x))\phi^l_{\tau}(x)dx.
\end{align*}
Since $u$ satisfies the weak form (\ref{2-1-1}) with $N(z)=|z|^p$ on $[0,T)$ 
and $\psi^l_{\tau}\in C_0^{\infty}([0,T)\times\mathbb{R}^d)$, 
by substituting $\psi^l_{\tau}$
into the test function in Definition \ref{def2-1},
and by using the identity $\{\partial_t(\psi_{\tau}^l)\}(0,x)=0$ for $x\in \mathbb{R}^d$,
the identities
\begin{align}
\label{3-1}
	I(\tau)+J(\tau)=&\int_{[0,\tau)\times B(\tau)}u\partial_t^2(\psi^l_{\tau})dxdt
		+\int_{[0,\tau)\times B(\tau)}(-u)\Delta(\psi^l_{\tau})dxdt\\
		&+\int_{[0,\tau)\times B(\tau)}(-u)b'(t)\psi_{\tau}^l dxdt
		+\int_{[0,\tau)\times B(\tau)}(-u)b(t)\partial_t(\psi_{\tau}^l)dxdt
		\notag\\
		=&:K_1+K_2+K_3+K_4\notag.		
\end{align}
hold for any $\tau\in (0,T)$.
We first estimate $K_1$.
Noting that $l/q-2\ge0$, by a direct calculation, the properties of the test-functions $\eta$, $\phi$ and H\"older's inequality, the estimates
\begin{align}
\label{3-2}
	K_1&\le l(l-1)\tau^{-2}\int_{[0,\tau)\times B(\tau)}
	|u|\eta^{l-2}_{\tau}\phi^l_{\tau}|(\partial_t\eta)(t/\tau)|^2dxdt\\
         &\ \ +l\tau^{-2}\int_{[0,\tau)\times B(\tau)}
         |u|\eta^{l-1}_{\tau}\phi^{l}_{\tau}|(\partial_t^2\eta)(t/\tau)|dxdt\notag\\
	&\le C\tau^{-2}\int_{[0,\tau)\times B(\tau)}|u|\psi^{l/p}_{\tau}dxdt
	\le C_1\tau^{(d+1)/q-2}\{I(\tau)\}^{1/p},\notag
\end{align}
hold for any $\tau\in (0,T)$, where $C_1$ is a positive constant dependent only on $l,d,p$ and $\eta$. Next we consider $K_2$. In the same manner as the proof of (\ref{3-2}), the inequalities
\begin{align}
\label{3-3}
	K_2&\le l(l-1)\tau^{-2}\int_{[0,\tau)\times B(\tau)}
	|u|\eta^l_{\tau}\phi^{l-2}_{\tau}|(\Delta\phi)(x/\tau)|dxdt\\
         &\ \ +l\tau^{-2}\int_{[0,\tau)\times B(\tau)}
         |u|\eta^{l}_{\tau}\phi^{l-1}_{\tau}|(\nabla\phi)(x/\tau)|^2dxdt\notag\\
	&\le C\tau^{-2}\int_{[0,\tau)\times B(\tau)}|u|\psi^{l/p}_{\tau}dxdt
	\le C_2\tau^{(d+1)/q-2}\{I(\tau)\}^{1/p}.\notag
\end{align}
hold for any $\tau\in (0,T)$, where $C_2$ is a positive constant dependent only on $l,d,p$ and $\phi$. Next we estimate $K_3$. Since $T$ is finite and
$b=b(t)$ is a $C^1$-function on $[0,\infty)$, by the properties of the test-functions $\eta$, $\phi$ and H\"older's inequality, the inequalities
\begin{equation}
\label{3-4}
	K_3\le \|b'\|_{L^{\infty}(0,T)}\int_{[0,\tau)\times B(\tau)}|u|\psi_{\tau}^{l/p}dxdt\le C_3\|b'\|_{L^{\infty}(0,T)}\tau^{(d+1)/q}\{I(\tau)\}^{1/p}.
\end{equation}
hold for any $\tau\in (0,T)$, where $C_3$ is a positive constant dependent only on $d,p$.
Finally we estimate $K_4$. Noting that $l/q-1\ge 0$, by the properties of the test-functions $\eta$, $\phi$ and H\"older's inequality, the estimates
\begin{align}
\label{3-7}
	K_4&\le l\|b\|_{L^{\infty}(0,T)}\tau^{-1}\int_{[0,\tau)\times B(\tau)}
	|u|\eta^{l-1}_{\tau}\phi^l_{\tau}|(\partial_t\eta)(t/\tau)|dxdt\\
         	&\le C\|b\|_{L^{\infty}(0,T)}\tau^{-1}\int_{[0,\tau)\times B(\tau)}|u|\psi^{l/p}_{\tau}dxdt
	\le C_4\|b\|_{L^{\infty}(0,T)}\tau^{(d+1)/q-1}\{I(\tau)\}^{1/p}\notag
\end{align}
hold for any $\tau\in (0,T)$, where $C_4$ is a positive constant dependent only on $l,d,p$ and $\eta$.
By combining the estimates (\ref{3-1})--(\ref{3-7}), the inequality
\begin{align}
\label{3-5}
	&I(\tau)+J(\tau)\le \\
	&\left\{(C_1+C_2)\tau^{(d+1)/q-2}+C_3\|b'\|_{L^{\infty}(0,T)}\tau^{(d+1)/q}+C_4\|b\|_{L^{\infty}(0,T)}\tau^{(d+1)/q-1}\right\}\{I(\tau)\}^{1/p}\notag
\end{align}
holds for any $\tau\in (0,T)$. Here since $p,q>1$ and $1/p+1/q=1$, the Young inequality 
\[
   ab\le \frac{a^p}{p}+\frac{b^q}{q}\ \ \text{for}\  a,b>0
\] holds. By combining this estimate and 
(\ref{3-5}), we have
\begin{equation}
\label{3-6}
	J(\tau)\le C_3^*\left\{\tau^{d+1-2q}+\|b'\|_{L^{\infty}(0,T)}^q\tau^{d+1}+\|b\|_{L^{\infty}(0,T)}^q\tau^{d+1-q}\right\},
\end{equation}
where $C_3^*$ is a positive constant dependent only on $d,p$ and $l$, which completes the proof of the lemma.
\end{proof}

By combining the condition that the initial data have a singularity at the origin $x=0$ and the test-function method (Lemma \ref{Lem 3-1}), we can derive the following estimate. The idea of the proof is found in Lemma 3.2 in \cite{IkIn15} or Lemma 3.2 in \cite{IkInpre}.
\begin{lemma}
\label{Lem 3-2}
Besides the assumptions in Lemma \ref{Lem 3-1}, we assume that the function $(u_0,u_1)$ satisfies (\ref{bc}) with $k<d$. Then the estimate
\begin{equation}
\label{3-0-2}
\int_{|x|\le \frac{1}{\tau}}|x|^{-k}\phi^l (x) dx\le C_3^*\tau^{k+1}\left\{\tau^{-2q}+\|b'\|_{L^{\infty}(0,T)}^q+\|b\|_{L^{\infty}(0,T)}^q\tau^{-q}\right\} 
\end{equation}
holds for any $\tau\in (0,T)$, 
where $C_3^*>0$ is the same constant which appears in Lemma \ref{Lem 3-1}.
\end{lemma}

\begin{proof}[Proof of Lemma \ref{Lem 3-2}]
By the definition of the function $J$, using changing variables and the assumption (\ref{bc}), we have
\[
           J(\tau) = \tau^{d}\int_{\mathbb{R}^{d}} (b(0)u_0+u_1)(\tau x) \phi^l(x) dx
                     \ge \tau ^{d-k}\int_{|x|\le \frac{1}{\tau}}|x|^{-k}\phi^l (x) dx
\]
for any $\tau>0$.
By combining Lemma \ref{Lem 3-1} and this estimate, the estimate (\ref{3-0-2}) is true for any $\tau\in (0,T)$, which completes the proof of the lemma.
\end{proof}

\subsection{Non-existence of global solutions for large data in the case $N(z)=\pm|z|^p$}
In this subsection, we give a proof of Proposition \ref{pro_ngw} and Theorem \ref{thm_lb}.

\begin{proof}[Proof of Proposition \ref{pro_ngw}]
Set $q:=\frac{p}{p-1}$, $l\in \mathbb{N}$ with $l\ge 2q+1$ and
\[
     \lambda_0:=C_3^*2^{k+1}\left\{2^{-2q}+\|b'\|_{L^{\infty}(0,2)}^q+\|b\|_{L^{\infty}(0,2)}^q2^{-q}\right\}
     \frac{d-k}{|S_{d-1}|}\left(\frac{1}{2}\right)^{k-d},
\]
where $C_3^*$ is given in Lemma \ref{Lem 3-1} and $|S_{d-1}|$ is the surface area of the unit sphere $S_{d-1}$ in $\mathbb{R}^d$. Then we can prove that for $\lambda\in (\lambda_0,\infty)$, the maximal existence time $T_{w}$ with $(u_0,u_1)=\lambda(a_0,a_1)$ is estimated as
\begin{equation}
\label{4-2-1}        
          T_w\le 2.
\end{equation}
Indeed, on the contrary, we assume that $T_w>2$. Let $u$ be a weak solution to (\ref{sdw}) on $[0,T_w)$.
We can apply Lemma \ref{Lem 3-2} with $(u_0,u_1)=\lambda(a_0,a_1)$ to obtain the estimate
\[
\lambda\int_{|x|\le \frac{1}{\tau}}|x|^{-k}\phi^l (x) dx\le C_3^*\tau^{k+1}\left\{\tau^{-2q}+\|b'\|_{L^{\infty}(0,T_w)}^q+\|b\|_{L^{\infty}(0,T_w)}^q\tau^{-q}\right\}, 
\]
for any $\tau\in (0,T_w)$. By this estimate with $\tau=2$, the inequality
\begin{equation}
\label{3-12-4}
\lambda\int_{|x|\le \frac{1}{2}}|x|^{-k}\phi^l (x) dx
\le C_3^*2^{k+1}\left\{2^{-2q}+\|b'\|_{L^{\infty}(0,2)}^q+\|b\|_{L^{\infty}(0,2)}^q2^{-q}\right\}
\end{equation}
holds. Since $k<d$, by the property of the test-function $\phi$, the identities
\begin{equation}
\label{3-13-4}
     \int_{|x|\le \frac{1}{2}}|x|^{-k}\phi^l (x) dx=\int_{|x|\le \frac{1}{2}}|x|^{-k}dx=\frac{|S_{d-1}|}{d-k}\left(\frac{1}{2}\right)^{d-k}=:C_4
\end{equation}
hold. By combining the estimates (\ref{3-12-4}) and (\ref{3-13-4}), the inequality
\[
    \lambda\le \lambda_0
\]
holds. This leads to a contradiction. Thus we have (\ref{4-2-1}).

Finally, we prove the estimate \eqref{Tes}.
It is obvious that the estimate
\begin{equation}
\label{3-14-4}
   \int_{|x|\le \frac{1}{\tau}}|x|^{-k}\phi^l (x) dx>\int_{|x|\le \frac{1}{2}}|x|^{-k}\phi^l (x) dx=C_4
\end{equation}
holds for any $\tau\in (0,2)$.
Let $\lambda\in (\lambda_0,\infty)$.
Then, \eqref{4-2-1} holds. By Lemma \ref{Lem 3-2} with the estimate (\ref{3-14-4}), the inequalities
\begin{equation}
\label{3-15-4}
   C_4\lambda<C_3^*\tau^{k+1-2q}\left\{1+2^{2q}\|b'\|_{L^{\infty}(0,2)}+2^q\|b\|_{L^{\infty}(0,2)}^q\right\}=:C_5\tau^{k+1-2q}
\end{equation}
hold for any $\tau\in (0,T_w)\subset (0,2]$. 
By the estimate $k<\frac{p+1}{p-1}$, the inequality $k+1-2q<0$ holds. Thus by the estimate (\ref{3-15-4}), the inequality
\[      \tau \le (C_4C_5^{-1})^{\frac{1}{k+1-2q}}\lambda^{-\frac{1}{\frac{p+1}{p-1}-k}}=:C_2^*\lambda^{-\frac{1}{\frac{p+1}{p-1}-k}}
\]
holds for any $\tau\in (0,T_w)$. Since $\tau$ is arbitrary in $(0,T_w)$, the above inequality implies (\ref{Tes}), which completes the proof of the proposition.
\end{proof}

\begin{proof}[Proof of Theorem \ref{thm_lb}]
Lemma \ref{lemso} implies the estimate
\begin{equation}
\label{3-17-4}
    T_+\le T_w.
\end{equation}
Since $p>1$ if $d=1,2$ and $p\in (1,p_1]$ if $d\ge 3$, the estimate $k<\frac{p+1}{p-1}$ holds, which allows us to apply Proposition \ref{pro_ngw}. Thus we that find there exists $\lambda_0>0$ such that for
$\lambda\in (\lambda_0,\infty)$, the estimate (\ref{Tes}) holds. By combining the estimates (\ref{3-17-4}) and (\ref{Tes}), the inequality (\ref{Btime}) holds for $\lambda\in (\lambda_0,\infty)$, which completes the proof of the theorem.
\end{proof}


\subsection{Non-existence of local solutions in the supercritical cases}
We give a proof of Proposition \ref{pro_non} and Theorem \ref{thm_non}.
\begin{proof}[Proof of Proposition \ref{pro_non}]
We only consider the case of $N(z)=|z|^p$, since the case of $N(z)=-|z|^p$ can be treated in a similar manner. On the contrary, we assume that there exists a weak solution $u$ to (\ref{sdw}) on $[0,T)$.
 By Lemma \ref{Lem 3-2}, the estimate
 \begin{equation}
 \label{3-17-4}
 \int_{|x|\le \frac{1}{\tau}}|x|^{-k}\phi^l (x) dx\le C_4\tau^{k+1-2q}
 \end{equation}
 holds for any $\tau\in (0,\min(1,T))$, where $C_4:=C_3(1+\|b'\|_{L^{\infty}(0,T)}^q+\|b\|_{L^{\infty}(0,T)}^q)$. Since $k<d$, by the properties of the test-function $\phi$, the inequalities
 \begin{align}
 \label{3-18-4}
       \int_{|x|\le \frac{1}{\tau}}|x|^{-k}\phi^l(x)dx&> \int_{|x|\le 1}|x|^{-k}\phi^l(x)dx
       > \int_{|x|\le \frac{1}{2}}|x|^{-k}dx\\
       &=\frac{|S_{d-1}|}{d-k}\left(\frac{1}{2}\right)^{d-k}=:C_5>0\notag
 \end{align}
 hold for any $\tau\in (0,1)$. Thus by combining the estimates (\ref{3-17-4}) and (\ref{3-18-4}), the estimate
\begin{equation}
\label{4-7}
             0<C_6:=C_5C_4^{-1}\le\tau^{k+1-2q}, 
\end{equation}
for any $\tau\in (0,\min(1,T))$. By the assumption $\frac{p+1}{p-1}<k$, the estimate $k+1-2q>0$ holds. 
Thus noting that $C_6$ is independent of $\tau$, we can take $\tau\in (0,\min(1,T))$ such as $\tau^{k+1-2q}<\frac{1}{2}C_6$, which contradicts (\ref{4-7}). This completes the proof of the proposition.
\end{proof}

\begin{proof}[Proof of Theorem \ref{thm_non}]
By the assumption $p>p_1$, the inequality $p>1+\frac{2}{d-1}$ holds. 
It follows that $(u_0,u_1)\in H^1(\mathbb{R}^d)\times L^2(\mathbb{R}^d)\subset (L^1_{loc}(\mathbb{R}^d))^2$. We also have $\frac{p+1}{p-1}<\frac{d}{2}$, which implies that there exists a function $(u_0,u_1)\in H^1(\mathbb{R}^d)\times L^2(\mathbb{R}^d)$ satisfying (\ref{bc}) with $\frac{p+1}{p-1}<k<\frac{d}{2}<d$. Thus we can apply Proposition \ref{pro_non} to conclude the proof of Theorem \ref{thm_non}.
\end{proof}


\appendix


\section{Proof of local well-posedness in the energy-subcritical case}
\label{a-1}
In this appendix, we give an outline of the proof of Proposition \ref{prop_le}. We apply the following transformation
\begin{align}%
\label{v}
	v(t,x) := \exp \left( \frac12 \int_0^t b(s)\,ds \right) u(t,x),\ \ t\in [0,T),\ x\in \mathbb{R}^d
\end{align}%
to the Cauchy problem \eqref{sdw} in order to eliminate the first order time-derivative of $u$.
A simple calculation gives
\begin{align*}%
	v_t(t,x) &= \left\{ \frac{b(t)}{2} u(t,x) + u_t(t,x) \right\}
		\exp \left( \frac12 \int_0^t b(s)\,ds \right),\\
	v_{tt}(t,x) &=
		\left[ \left\{ \frac{b^{\prime}(t)}{2} + \frac{b(t)^2}{4} \right\}u(t,x)
			+ b(t) u_t(t,x) + u_{tt}(t,x) \right] \exp \left( \frac12 \int_0^t b(s)\,ds \right).
\end{align*}%
We find that $u$ is the solution to (\ref{sdw}) if and only if $v$ is the solution to the Cauchy problem of the semilinear wave equation
\begin{align}%
\label{sdwv}
	\left\{ \begin{array}{ll}
		v_{tt} - \Delta v = F(v),
		&t\in [0,T), x\in \mathbb{R}^d,\\
		v(0,x) = v_0(x),\ v_t(0,x) = v_1(x),
		&x\in \mathbb{R}^d,
		\end{array}\right.
\end{align}%
where
$F(v) := F_1(v) +F_2(v)$,
\begin{align*}%
	F_1(v) := \left\{ \frac{b^{\prime}(t)}{2} + \frac{b(t)^2}{4} \right\}v,
	\quad
	F_2(v) := \exp \left( \frac{1}{2} \int_0^t b(s)\,ds \right)
			N\left( \exp \left( -\frac{1}{2} \int_0^t b(s)\,ds \right) v \right)
\end{align*}%
and
\begin{align*}%
	v_0(x) := u_0(x),\quad
	v_1(x) :=  \frac{b(0)}{2}u_0(x) + u_1(x).
\end{align*}%
We note that the right hand side of (\ref{sdwv}) does not include the time-derivative $\partial_t v$. Moreover, since the nonlinear function $N$ satisfies (\ref{n}) with $p\ge 1$ and $b$ is non-negative, the estimate
\[
    |F_2(z)|\le C_N\exp\left(-\frac{p-1}{2}\int_0^tb(s)dx\right)|z|^p\le C_N|z|^p
\]
holds for any $z\in \mathbb{R}$. Thus since $(v_0,v_1)$ belongs to $H^1(\mathbb{R}^d)\times L^2(\mathbb{R}^d)$ and we consider local existence of solutions to (\ref{sdw}), we can apply the argument of the proof of Lemma 4.1 in \cite{MNO02}, the estimates (2.15) and (2.16) in \cite{MN02} or the fundamental theorem in \cite{Ka92}, to obtain local well-posedness in $H^1(\mathbb{R}^d)\times L^2(\mathbb{R}^d)$ to (\ref{sdw}).

\section*{Acknowledgement}
The authors are deeply grateful to Professor Kenji Nakanishi for giving us an idea for estimating $L^2$-norm of solutions to (\ref{sdw}) in Theorem \ref{thm_ge}.
The authors would like to thank the referee for careful reading of the manuscript
and giving constructive comments.
This work is supported by Grant-in-Aid for JSPS Fellows 26$\cdot$1884 and Grant-in-Aid for Young Scientists (B) 15K17571 and 16K17625.

\end{document}